\documentclass[12pt]{article}
\usepackage[centertags]{amsmath}
\usepackage{geometry}
\usepackage{amsfonts}
\usepackage{amssymb}
\usepackage{amsthm}
\usepackage{newlfont}
\usepackage{graphicx}
\usepackage{epstopdf}
 \usepackage[]{titletoc}

\newlength{\defbaselineskip}
\newcommand{\setlinespacing}[1]%
           {\setlength{\baselineskip}{#1 \defbaselineskip}}

\theoremstyle{plain}
\newtheorem{thm}{Theorem}[section]
\newtheorem{cor}[thm]{Corollary}
\newtheorem{lem}[thm]{Lemma}
\newtheorem{prop}[thm]{Proposition}
\newtheorem{exam}[thm]{Example}
\newtheorem{rem}[thm]{Remark}
\newtheorem{Def}[thm]{Definition}

\newcommand{\CC}{\mathbb{C}}
\newcommand{\DD}{\mathbb{D}}
\newcommand{\TT}{\mathbb{T}}
\newcommand{\HH}{\mathbb{H}}
\newcommand{\RR}{\mathbb{R}}
\newcommand{\OOO}{\mathcal{O}}

\newcommand{\KKK}{\mathcal{K}}

\newcommand{\I}{\uppercase\expandafter{\romannumeral1}}
\newcommand{\II}{\uppercase\expandafter{\romannumeral2}}
\newcommand{\III}{\uppercase\expandafter{\romannumeral3}}
\newcommand{\IV}{\uppercase\expandafter{\romannumeral4}}
\newcommand{\V}{\uppercase\expandafter{\romannumeral5}}
\newcommand{\VI}{\uppercase\expandafter{\romannumeral6}}
\newcommand{\Log}{\mathrm{Log}}

\makeatletter\@addtoreset{equation}{section} \makeatother
\begin{document}
\title {A Sharp Inequality of Hardy-Littlewood Type Via Derivatives}
 \author{Hui Dan, Kunyu Guo and Yi Wang }
\vskip2mm

 \maketitle \noindent\textbf{Abstract:}

In this paper we consider a generalized version of Carleman's inequality. An equivalent version of it states that $\|f\|_{A_\alpha^{2\alpha}}\leq\|f\|_{H^2}$, where $f$ is a holomorphic function and $\alpha>1$. If the norms $\|f\|_{A_\alpha^{2\alpha}}$ are decreasing in $\alpha$, then the inequality holds for $f$. For a dense set of functions, we calculate the derivative of the norms $\|f\|_{A_\alpha^{2\alpha}}$ in $\alpha$ and give sufficient conditions for this derivative to be non-positive. As an application, we prove the inequality for linear combinations of two reproducing kernels. Some numerical evidences are also provided.



\section{Introduction}
In this paper we consider a sharp inequality concerning the weighted Bergman norms and the Hardy norm on the unit disc $\DD$. Recall that for $0<p<\infty$, the Hardy space $H^p$ consists of all holomorphic functions $f$ on $\DD$ such that 
\[
\|f\|_{H^p}:=\sup_{0<r<1}\bigg(\int_0^{2\pi}|f(re^{i\theta})|^p\frac{\mathrm{d}\theta}{2\pi}\bigg)^{\frac{1}{p}}<\infty.
\]
For $\alpha>1$, the weighted Bergman space $A_\alpha^p$ consists of all holomorphic functions $f$ on $\DD$ such that 
\[
\|f\|_{A_\alpha^p}:=\bigg(\int_{\DD}|f(z)|^p(\alpha-1)(1-|z|^2)^\alpha \mathrm{d}\mu(z)\bigg)^{\frac{1}{p}}<\infty,
\]
where $\mathrm{d}\mu(z)=(1-|z|^2)^{-2}\frac{\mathrm{d}x\mathrm{d}y}{\pi}$ is the M\"{o}bius invariant measure of the unit disc.
The inequality we are considering in this paper is the following.

\medskip

\noindent\textbf{Conjecture 1.} For any $0<p\leq2$ and any $f \in H^p$,
\begin{equation}\label{eqn: conj. 1}
 \|f\|_{A_{2/p}^2}\leq\|f\|_{H^p}.
\end{equation}

\medskip

In the case when $p=1$, \eqref{eqn: conj. 1} is the Carleman's inequality (cf. \cite{Vuk03}). For $p=1/k$ where $k$ is any positive integer, Burbea \cite{Bur87} showed that \eqref{eqn: conj. 1} holds true. If one releases the restriction on the controlling constants, that is, if one asks whether 
\begin{equation}\label{eqn: conj. 1 with constant}
\|f\|_{A_{2/p}^2}\leq C\|f\|_{H^p}
\end{equation}
 for some constant $C$, then using interpolation techniques, Brevig, Ortega-Cerd\`{a}, Seip and Zhao has proved that \eqref{eqn: conj. 1 with constant} holds for $0<p<1$ and $C$ as close to $1$ as $C=(2/(e\log2))^{1/2}=1.030279\ldots$. In \cite{BOSZ}, the authors also gave and discussed about several interesting related conjectures and questions. 

In the case when $p=1$, inequality \eqref{eqn: conj. 1} becomes
\[
\|f\|_{A_2^2}\leq\|f\|_{H^1}.
\]
This is known as the Carleman's inequality. In 1921, Carleman \cite{Carl21} proved this inequality and used it to give the first complex-analytic proof of the famous isoperimetric theorem. For a different purpose, in 1932, Hardy and Littlewood showed that $H^p\subset A_2^{2p}$ (in particular, $H^1\subset A_2^2$) in \cite{HL32}.
See \cite{Vuk03} for an excellent exposition of the relation between the two problems. Various generalizations were proved, for example, in \cite{Aro50}\cite{Bur831}\cite{Bur832}\cite{Bur87}\cite{MP84}\cite{War61}.

In recently years, Inequality \eqref{eqn: conj. 1} has regained attention because of its application in number theory. Via an iterating process \cite{Bay02} \cite{Hel06}, contractive inequalities like \eqref{eqn: conj. 1} may ``lift'' multiplicatively to interesting inequalities for Hardy spaces on the infinite-dimensional torus, which in turn, by the Bohr transform, translates into inequalities of Dirichlet polynomials\cite{BBHOP}\cite{BBSSZ}\cite{BHS}.

Next, let us go to the technical side. An immediate observation is that for Conjecture 1, it suffices to consider any outer function $f$, because multiplying an inner function on $f$ does not make a difference on the right hand side of \eqref{eqn: conj. 1}, but makes the left hand side smaller. For an outer function, one can consider its powers. By replacing $f$ with $f^\alpha$ where $\alpha=2/p$, it is easy to show that Conjecture 1 is equivalent to the following (cf. \cite{BOSZ}).

\medskip

\noindent\textbf{Conjecture 2.} For any $\alpha>1$ and any $f\in H^2$,
\begin{equation}\label{eqn: conj. 2}
\|f\|_{A_\alpha^{2\alpha}}\leq\|f\|_{H^2}.
\end{equation}

\medskip

Then Burbea's result \cite{Bur87} is equivalent to that \eqref{eqn: conj. 2} holds when $\alpha$ is any integer that is greater than $1$. A straight-forward proof was given in \cite[Corollary 3]{BOSZ}. In the case when $\alpha$ is not an integer, the problem becomes very hard. The following computation may give us a clue. In the case when $\alpha>1$ is an integer, and suppose that $f$ is an outer function, $f=\sum_{n=0}^\infty a_nz^n$ and $a_0=1$. Then one can compute that
\begin{align}\label{eqn: computation 1}
\nonumber&\|f\|_{H^2}^{2\alpha}-\|f\|_{A_\alpha^{2\alpha}}^{2\alpha}\\
=&\frac{1}{2}\sum\limits_{N=0}^\infty {N+\alpha-1\choose N}^{-1}\sum\limits_{k,l=1}^N\sum\limits_{\substack{n_1+\cdots+n_k=N, n_i\geq1\\ m_1+\cdots+m_l=N, m_i\geq1}}{\alpha\choose k}{\alpha\choose l}
\bigg|a_{n_1}\cdots a_{n_k}-a_{m_1}\cdots a_{m_l}\bigg|^2.
\end{align}
This gives an alternative proof of Burbea's result. For non-integer valued $\alpha$, we have the same equation (under some convergence assumption). However, the coefficients $\alpha\choose k$ and $\alpha\choose l$ may be negative. Similar obstructions occur when one tries to extend other proofs of Burbea's result to a non-integer valued $\alpha$.

In \cite{BOSZ}, the authors gave several related conjectures (including the Conjectures 1 and 2 above) and questions. In particular, in \cite[Question 1]{BOSZ}, they asked whether $\|f\|_{A_\alpha^{2\alpha}}^{2\alpha}$ is non-increasing in the parameter $\alpha$, for an outer function $f$ with $\|f\|_{H^2}=1$. A positive answer to the question above will lead to a positive answer to Conjecture 2. In this paper, we will mainly consider the following similar question, which allows us to drop the assumption ``$\|f\|_{H^2}=1$'' (see Remark \ref{rem: f hardy norm equals 1}).

\medskip

\noindent\textbf{Question 3.} Suppose that $f$ is an outer function. For $\alpha>1$, denote 
\[
N_f(\alpha)=\|f\|_{A_\alpha^{2\alpha}}=\bigg(\int_\DD|f(z)|^{2\alpha}(\alpha-1)(1-|z|^2)^\alpha \mathrm{d}\mu(z)\bigg)^{\frac{1}{2\alpha}}.
\]
Is it true that 
\[
\frac{\partial}{\partial\alpha}N_f(\alpha)\leq0
\]
for all $\alpha>1$?

\medskip

In Section \ref{sec: a discrete formula}, we will first show that a positive answer to Question 3 implies Conjecture 2. Then we will give a discrete formula of $\frac{\partial}{\partial\alpha}N_f(\alpha)$ for a dense set of functions. Based on the formula, in Section \ref{sec: sufficient conditions}, we give some sufficient conditions for $\frac{\partial}{\partial\alpha}N_f(\alpha)$ to be non-positive. As an application, in Section \ref{sec: an application}, we obtain the following result.
\begin{thm}[Theorem \ref{thm: application} ]
	Suppose $f\in H^2$ and $f=\eta F$, where $\eta$ is inner and $F$ has no zeros in $\DD$. Suppose
\[
F^\alpha=c_1K_{w_1,\alpha}+c_2K_{w_2,\alpha}.
\] 
for some $\alpha>1$, and $\mathbf c\in\CC^2, \mathbf w\in\DD^2$.
	Then for any $1\leq\beta\leq\alpha$, we have 
	\begin{equation}
	\|F\|_{A_\alpha^{2\alpha}}\leq\|F\|_{A_\beta^{2\beta}}.
	\end{equation}
	Equality holds if and only if $F^\alpha=cK_{w,\alpha}$ for some $c\in\CC$ and $w\in\DD$.
	As a consequence, we have
	\begin{equation}
	\|f\|_{A_\alpha^{2\alpha}}\leq\|f\|_{H^2}.
	\end{equation}
	Equality holds if and only if $f=cK_{w,1}$ for some $c\in\CC$ and $w\in\DD$.
\end{thm}
 Some further remarks and numerical evidences are provided in Section \ref{sec: some evidences}.

\section{A Discrete Formula}\label{sec: a discrete formula}
It is well-known that 
\begin{equation}\label{eqn: lim is hardy norm classic}
\lim_{\alpha\to1+}\|f\|_{A_\alpha^p}=\|f\|_{H^p}
\end{equation}
for $p>0$ and $f\in H^p$ \cite{Zhu04}. Similarly, we have the following lemma.

\begin{lem}\label{lem: limit is hardy norm}
	Let $\OOO^\ast$ denote the set of holomorphic functions $f$ defined in some open neighborhood of $\overline{\DD}$ such that $f(z)\neq0$ for all $z\in\overline{\DD}$. Then for any $f\in\OOO^\ast$ and any $p>0$,
	\begin{equation}\label{eqn: lim is hardy norm}
	\lim_{\alpha\to1+}\|f\|_{A_\alpha^{\alpha p}}=\|f\|_{H^p}.
	\end{equation}
	As a consequence, if $f\in\OOO^\ast$ and $\frac{\partial}{\partial\alpha}N_f(\alpha)\leq0$, $\forall\alpha>1$, then $\|f\|_{A_{\alpha}^{2\alpha}}\leq\|f\|_{H^2}$, $\forall\alpha>1$.
\end{lem}

\begin{proof}
	Without loss of generality, we assume $\|f\|_{H^p}=1$.
	Since $f\in\OOO^\ast$, there exists $C>c>0$ such that $c\leq|f(z)|\leq C$ for $z\in\DD$. It is easy to find a constant $M>0$ such that $|x^\alpha-x|\leq M(\alpha-1)$ for $x\in[c^p, C^p]$ and $\alpha\in(1,2)$. By \eqref{eqn: lim is hardy norm classic}, 
	\[
	\lim_{\alpha\to1+}\int_\DD|f(z)|^p(\alpha-1)(1-|z|^2)^\alpha \mathrm{d}\mu(z)=1.
	\]
	On the other hand, since $|f(z)|^p\in[c^p, C^p]$ for all $z\in\DD$, we have
	\begin{eqnarray*}
	&&\bigg|\int_\DD|f(z)|^p(\alpha-1)(1-|z|^2)^\alpha \mathrm{d}\mu(z)-\int_\DD|f(z)|^{\alpha p}(\alpha-1)(1-|z|^2)^\alpha \mathrm{d}\mu(z)\bigg|\\
	&\leq&M(\alpha-1)\int_\DD(\alpha-1)(1-|z|^2)^\alpha \mathrm{d}\mu(z)\\
	&=&M(\alpha-1)\to0,\quad \alpha\to1+.
	\end{eqnarray*}
Therefore 
\[
\|f\|_{A_\alpha^{p\alpha}}^{p\alpha}\to1,\quad \alpha\to1+.
\]
So $\lim_{\alpha\to1+}\|f\|_{A_\alpha^{p\alpha}}=1$. The rest of the lemma is obvious. This completes the proof.
\end{proof}

\begin{rem}
	From Lemma \ref{lem: limit is hardy norm}, it is easy to see that a positive answer to Question 3 implies Conjecture 2. The statement that $\frac{\partial}{\partial\alpha}N_f(\alpha)\leq0$ looks like a stronger statement than Conjecture 2. However, we are still optimistic enough to expect a positive answer. One of the evidences is the following. In \cite{BOSZ}, the authors proved a lemma (\cite[Lemma 2]{BOSZ}) which implies $\|f\|_{A_{k\alpha}^{2k\alpha}}\leq\|f\|_{A_\alpha^{2\alpha}}$ for any positive integer $k$, and used this lemma to prove \eqref{eqn: conj. 2} in the case when $\alpha>1$ is an integer. In Section \ref{sec: some evidences}, we will also provide some numerical evidences that support a positive answer to Question 3.	
\end{rem}

Denote $\Log z$ the single-valued branch of $\log z$ on $\CC\backslash\{z\in\RR: z\leq0\}$ such that $\Log 1=0$.
By direct computation, we have

\begin{prop}\label{prop: derivative}
	If $f\in\OOO^\ast$, then
	\begin{equation}\label{eqn: derivative}
	\frac{\partial}{\partial\alpha}N_f(\alpha)=\frac{\alpha-1}{2\alpha}N_f^{1-2\alpha}(\alpha)\bigg(-\frac{2}{\alpha-1}N_f^{2\alpha}(\alpha)\Log N_f(\alpha)+\frac{1}{(\alpha-1)^2}N_f^{2\alpha}(\alpha)+\I_f(\alpha)\bigg),
	\end{equation}
	where
	\begin{equation}\label{eqn: derivative I}
	\I_f(\alpha)=\int_{\DD}|f(z)|^{2\alpha}(1-|z|^2)^\alpha\Log\big(|f(z)|^2(1-|z|^2)\big)\mathrm{d}\mu(z).
	\end{equation}
\end{prop}

Suppose $f$ is an outer function, then $f^\alpha$ makes sense and is also an outer function. We have the equation 
\[
\|f\|_{A_\alpha^{2\alpha}}^{\alpha}=\|f^\alpha\|_{A_\alpha^2}.
\]
Also, from \eqref{eqn: derivative} and \eqref{eqn: derivative I}, we see that the value of $\frac{\partial}{\partial\alpha}N_f(\alpha)$ depends only on the function $f^\alpha$. This allows us to consider $f^\alpha$ instead of $f$ and take advantage of the reproducing kernel Hilbert space structure of $A_\alpha^2$.

The main goal of this section is to prove the following theorem. 

\begin{thm}\label{thm: discrete formula}
	Suppose $f\in\OOO^\ast$ and $\alpha>1$. Suppose there exists $k$ points, $w_1,\cdots,w_k$ in $\DD$ and $k$ numbers $c_1,\cdots,c_k\in\CC$ such that
	\[
	f^\alpha(z)=\sum_{i=1}^kc_i\frac{1}{(1-\overline{w_i}z)^\alpha},\quad z\in\DD.
	\]
	Then
	\begin{equation}\label{eqn: discrete derivative}
	\frac{\partial}{\partial\alpha}N_f(\alpha)=\frac{1}{2\alpha^2}N_f^{1-2\alpha}(\alpha)D_f(\alpha),
	\end{equation}
	where
	\begin{equation}\label{eqn: Df=...}
	D_f(\alpha)=\sum_{i,j=1}^kc_i\overline{c_j}\frac{1}{(1-\overline{w_i}w_j)^\alpha}\bigg(\overline{\log f^\alpha(w_i)}+\log f^\alpha(w_j)+\alpha\Log(1-\overline{w_i}w_j)-2\alpha\Log N_f(\alpha)\bigg).
	\end{equation}
	For $\log f^\alpha(w_i)$, we fix a holomorphic function $g$ such that $f^\alpha(z)=e^{g(z)}$ and let $\log f^\alpha(w_i)=g(w_i)$. (since $f^\alpha$ is outer, such function $g$ exists.)
\end{thm}

Before proving Theorem \ref{thm: discrete formula}, let us use an example to illustrate our idea.

\begin{exam}\label{example}
It is well-known (and also implied by the proof of \cite[Corollary 3]{BOSZ}) that for integer-valued $\alpha$, the equation in \eqref{eqn: conj. 2} holds if and only if $f(z)=\frac{c}{1-\overline{w}z}$ for some $c\in\CC$ and $w\in\DD$. Indeed, if $f(z)=\frac{1}{1-\overline{w}z}$, then one can compute directly that $\frac{\partial}{\partial\alpha}N_f(\alpha)\equiv0$.
	
	By Proposition \ref{prop: derivative}, in order to compute $\frac{\partial}{\partial\alpha}N_f(\alpha)$, one needs to find out $\I_f(\alpha)$. Applying the M\"{o}bius transform $\lambda=\varphi_w(z)=\frac{w-z}{1-\overline{w}z}$, we get
	\begin{eqnarray*}
		\I_f(\alpha)&=&\int_\DD\bigg(\frac{1-|z|^2}{|1-\overline{w}z|^2}\bigg)^\alpha\Log\frac{1-|z|^2}{|1-\overline{w}z|^2}\mathrm{d}\mu(z)\\
		&=&\int_{\DD}\bigg(\frac{1-|\lambda|^2}{1-|w|^2}\bigg)^\alpha\Log\frac{1-|\lambda|^2}{1-|w|^2}\mathrm{d}\mu(\lambda)\\
		&=&(1-|w|^2)^{-\alpha}\int_\DD(1-|\lambda|^2)^{\alpha-2}\Log(1-|\lambda|^2)\frac{\mathrm{d}m(\lambda)}{\pi}\\
		&&-(1-|w|^2)^{-\alpha}\Log(1-|w|^2)\int_\DD(1-|\lambda|^2)^{\alpha-2}\frac{\mathrm{d}m(\lambda)}{\pi}.
	\end{eqnarray*}
Here $dm$ denotes the Lebesgue measure.
Using the polar coordinates and applying an integration by parts, we have
\[
\int_\DD(1-|\lambda|^2)^{\alpha-2}\Log(1-|\lambda|^2)\frac{\mathrm{d}m(\lambda)}{\pi}=-(\alpha-1)^{-2}.
\]
Similar computations give
\[
\int_\DD(1-|\lambda|^2)^{\alpha-2}\frac{\mathrm{d}m(\lambda)}{\pi}=\frac{1}{\alpha-1}
\]
and
\[
N_f^{2\alpha}(\alpha)=(1-|w|^2)^{-\alpha}.
\]
Thus
\[
I_f(\alpha)=-(\alpha-1)^{-2}(1-|w|^2)^{-\alpha}-(\alpha-1)^{-1}(1-|w|^2)^{-\alpha}\Log(1-|w|^2).
\]
From this and Proposition \ref{prop: derivative}, it is easy to see that $\frac{\partial}{\partial\alpha}N_f(\alpha)\equiv0$.
\end{exam}

By Example \ref{example}, $\frac{\partial}{\partial\alpha}N_f(\alpha)=0$ whenever $f^\alpha$ is a constant multiple of a reproducing kernel of $A_\alpha^2$. The linear span of reproducing kernels form a dense set in $A_\alpha^2$. This explains our reason of considering such functions in Theorem \ref{thm: discrete formula}.

Next, let us give the proof of Theorem \ref{thm: discrete formula}.
\begin{proof}[\textbf{Proof of Theorem \ref{thm: discrete formula}}]
	Suppose $f\in\OOO^\ast$ and 
\[
f^\alpha(z)=\sum_{i=1}^kc_i\frac{1}{(1-\overline{w_i}z)^\alpha}.
\]
As in Example \ref{example}, in order to calculate $\frac{\partial}{\partial\alpha}N_f(\alpha)$, we need to find out $\I_f(\alpha)$ as defined in \eqref{eqn: derivative I}. Compared with Example \ref{example}, the main difficulty here is that we can not use the M\"{o}bius transform. We will get around by applying the Stoke's Theorem and the Residue Theorem.

Let 
\begin{equation}\label{eqn: II}
\II=\frac{1}{\alpha}\int_\DD|f^\alpha(z)|^2(1-|z|^2)^\alpha\Log|f^\alpha(z)|^2\mathrm{d}\mu(z)
\end{equation}
and
\begin{equation}\label{eqn: III}
\III=\int_\DD|f^\alpha(z)|^2(1-|z|^2)^\alpha\Log(1-|z|^2)\mathrm{d}\mu(z).
\end{equation}

By \eqref{eqn: derivative I}, it is easy to see that 
\begin{equation}\label{I=II+III}
\I_f(\alpha)=\II+\III.	
\end{equation}

Taking advantage of the fact that $\log f^\alpha(z)$ is a holomorphic function in $\DD$, we have
\begin{eqnarray}
\nonumber \II&=&\frac{2}{\alpha}\mathrm{Re}\bigg(\int_\DD|f^\alpha(z)|^2(1-|z|^2)^\alpha\log f^\alpha(z)\mathrm{d}\mu(z)\bigg)\\ \nonumber
&=&\frac{2}{\alpha(\alpha-1)}\mathrm{Re}\sum_{j=1}^k\overline{c_j}\int_\DD\frac{1}{(1-w_j\overline{z})^\alpha}f^\alpha(z)\log f^\alpha(z)(\alpha-1)(1-|z|^2)^\alpha \mathrm{d}\mu(z)\\
\nonumber&=&\frac{2}{\alpha(\alpha-1)}\mathrm{Re}\sum_{j=1}^k\overline{c_j}f^\alpha(w_j)\log f^\alpha(w_j)\\
\label{eqn: II =...}&=&\frac{1}{\alpha(\alpha-1)}\sum_{i,j=1}^kc_i\overline{c_j}\frac{1}{(1-\overline{w_i}w_j)^\alpha}\bigg(\overline{\log f^\alpha(w_i)}+\log f^\alpha(w_j)\bigg).
\end{eqnarray}

Next, we calculate $\III$.
\begin{eqnarray}
	\nonumber \III&=&\int_\DD|f^\alpha(z)|^2(1-|z|^2)^\alpha\Log(1-|z|^2)\mathrm{d}\mu(z)\\
	\nonumber&=&\sum_{i,j=1}^kc_i\overline{c_j}\int_\DD\frac{1}{(1-\overline{w_i}z)^\alpha}\frac{1}{(1-\overline{z}w_j)^\alpha}(1-|z|^2)^\alpha\Log(1-|z|^2)\mathrm{d}\mu(z)\\
	\label{III=sumIIIij}&=&\sum_{i,j=1}^kc_i\overline{c_j}\III_{ij}.
\end{eqnarray}
Here
\begin{equation}\label{eqn: IIIij}
\III_{ij}=\int_\DD\frac{1}{(1-\overline{w_i}z)^\alpha}\frac{1}{(1-\overline{z}w_j)^\alpha}(1-|z|^2)^\alpha\Log(1-|z|^2)\mathrm{d}\mu(z).
\end{equation}

For $i,j=1,\cdots,k$, define
\begin{equation}\label{eqn: IVij}
\IV_{ij}=\int_\DD\frac{1}{(1-\overline{w_i}z)^\alpha}\frac{1}{(1-\overline{z}w_j)^\alpha}(1-|z|^2)^\alpha\Log\frac{1-|z|^2}{|1-\overline{z}w_j|^2}\mathrm{d}\mu(z)
\end{equation}
and
\begin{equation}\label{eqn: Vij}
\V_{ij}=\int_\DD\frac{1}{(1-\overline{w_i}z)^\alpha}\frac{1}{(1-\overline{z}w_j)^\alpha}(1-|z|^2)^\alpha\Log|1-\overline{z}w_j|^2\mathrm{d}\mu(z).
\end{equation}
Then 
\begin{equation}\label{IIIij=IVij+Vij}
\III_{ij}=\IV_{ij}+\V_{ij}.
\end{equation}

Since $\overline{\Log z}=\Log\overline{z}$, we have
\begin{eqnarray}
	\nonumber \V_{ij}&=&\int_\DD\frac{1}{(1-\overline{w_i}z)^\alpha}\frac{1}{(1-\overline{z}w_j)^\alpha}\Log(1-\overline{z}w_j)(1-|z|^2)^\alpha \mathrm{d}\mu(z)\\
	\nonumber&&+\int_\DD\frac{1}{(1-\overline{w_i}z)^\alpha}\frac{1}{(1-\overline{z}w_j)^\alpha}\Log(1-\overline{w_j}z)(1-|z|^2)^\alpha \mathrm{d}\mu(z)\\
	\nonumber&=&\frac{1}{\alpha-1}\bigg(\frac{1}{(1-\overline{w_i}w_j)^\alpha}\Log(1-\overline{w_i}w_j)+\frac{1}{(1-\overline{w_i}w_j)^\alpha}\Log(1-|w_j|^2)\bigg)\\
	\label{ Vij=...}&=&\frac{1}{\alpha-1}\frac{1}{(1-\overline{w_i}w_j)^\alpha}\bigg(\Log(1-\overline{w_i}w_j)+\Log(1-|w_j|^2)\bigg).
\end{eqnarray}

It remains to calculate $\IV_{ij}$. Let
\[
\varphi_{ij}(z)=\frac{1}{\alpha-1}\frac{1}{w_j-z}\frac{(1-|z|^2)^{\alpha-1}}{(1-\overline{z}w_j)^{\alpha}}\frac{1}{(1-\overline{w_i}z)^\alpha},\quad z\in\DD, z\neq w_j.
\]
and
\[
\psi_{ij}(z)=\Log\frac{1-|z|^2}{|1-\overline{z}w_j|^2}.
\]
By direct computation, we have
\[
\bar{\partial}\varphi_{ij}(z)=\frac{1}{(1-\overline{w_i}z)^\alpha}\frac{(1-|z|^2)^{\alpha-2}}{(1-\overline{z}w_j)^\alpha}
\]
and
\[
\bar{\partial}\psi_{ij}(z)=\frac{w_j-z}{(1-|z|^2)(1-\overline{z}w_j)}.
\]
Therefore
\begin{equation}\label{eqn: IVij=...1}
\IV_{ij}=\frac{1}{\pi}\int_\DD\bar{\partial}\varphi_{ij}(z)\psi_{ij}(z)\mathrm{d}x\mathrm{d}y.
\end{equation}

For any $\varepsilon>0$ sufficiently small, define
\[
\DD_{\varepsilon, j}:=\{z\in\DD: |z-w_j|>\varepsilon\}.
\]
Define the one-form $\omega=\varphi_{ij}(z)\psi_{ij}(z)\mathrm{d}z$. Then
\[
\mathrm{d}\omega=-(\bar{\partial}\varphi_{ij}\psi_{ij}+\varphi_{ij}\bar{\partial}\psi_{ij})\mathrm{d}z\wedge \mathrm{d}\bar{z}=2\sqrt{-1}(\bar{\partial}\varphi_{ij}\psi_{ij}+\varphi_{ij}\bar{\partial}\psi_{ij})\mathrm{d}x\wedge \mathrm{d}y.
\]
Applying the Stokes's Theorem on $\DD_{\varepsilon,j}$, we get 
\begin{eqnarray*}
&&\frac{1}{\pi}\int_{\DD_{\varepsilon,j}}(\bar{\partial}\varphi_{ij}\psi_{ij}+\varphi_{ij}\bar{\partial}\psi_{ij})\mathrm{d}x\wedge \mathrm{d}y\\
&=&\frac{1}{2\pi\sqrt{-1}}\bigg(\int_\TT \varphi_{ij}\psi_{ij}\mathrm{d}z-\int_{\{z: |z-w_j|=\varepsilon\}}\varphi_{ij}\psi_{ij}\mathrm{d}z\bigg)\\
&=&-\frac{1}{2\pi\sqrt{-1}}\int_{\{z: |z-w_j|=\varepsilon\}}\varphi_{ij}\psi_{ij}\mathrm{d}z.
\end{eqnarray*}
The second equality is because $\varphi_{ij}\psi_{ij}=0$ on the unit circle $\TT$.

Therefore
\begin{eqnarray}
	\nonumber \IV_{ij}&=&\lim_{\varepsilon\to0}\frac{1}{\pi}\int_{\DD_{\varepsilon,j}}\bar{\partial}\varphi_{ij}\psi_{ij}\mathrm{d}x\wedge \mathrm{d}y\\
	\nonumber&=&-\lim_{\varepsilon\to0}\bigg(\frac{1}{\pi}\int_{\DD_{\varepsilon,j}}\varphi_{ij}\bar{\partial}\psi_{ij}\mathrm{d}x\wedge \mathrm{d}y+\frac{1}{2\pi\sqrt{-1}}\int_{\{z: |z-w_j|=\varepsilon\}}\varphi_{ij}\psi_{ij}\mathrm{d}z\bigg)\\
	\nonumber&=&-\int_\DD\frac{1}{\alpha-1}\frac{(1-|z|^2)^\alpha}{(1-\overline{w_i}z)^\alpha(1-\overline{z}w_j)^\alpha}\mathrm{d}\mu(z)\\
	\nonumber&&-\lim_{\varepsilon\to0}\frac{1}{2\pi\sqrt{-1}}\int_{\{z: |z-w_j|=\varepsilon\}}\frac{1}{\alpha-1}\frac{1}{w_j-z}\frac{(1-|z|^2)^{\alpha-1}}{(1-\overline{z}w_j)^{\alpha-1}}\frac{1}{(1-\overline{w_i}z)^\alpha}\Log\frac{1-|z|^2}{|1-\overline{z}w_j|^2}\mathrm{d}z\\
	\label{IVij=...}&=&-\frac{1}{(\alpha-1)^2}\frac{1}{(1-\overline{w_i}w_j)^\alpha}+\VI_{ij},
\end{eqnarray}
where
\[
\VI_{ij}=-\lim_{\varepsilon\to0}\frac{1}{2\pi\sqrt{-1}}\int_{\{z: |z-w_j|=\varepsilon\}}\frac{1}{\alpha-1}\frac{1}{w_j-z}\frac{(1-|z|^2)^{\alpha-1}}{(1-\overline{z}w_j)^{\alpha-1}}\frac{1}{(1-\overline{w_i}z)^\alpha}\Log\frac{1-|z|^2}{|1-\overline{z}w_j|^2}\mathrm{d}z.
\]
To calculate $\VI_{ij}$, notice that 
\[
\frac{(1-|z|^2)^{\alpha-1}}{(1-\overline{z}w_j)^{\alpha-1}}\frac{1}{(1-\overline{w_i}z)^\alpha}\Log\frac{1-|z|^2}{|1-\overline{z}w_j|^2}\to -\frac{1}{(1-\overline{w_i}w_j)^\alpha}\Log(1-|w_j|^2), \quad z\to w_j.
\]
Standard estimates will give us
\begin{eqnarray}
\nonumber \VI_{ij}&=&	\frac{1}{\alpha-1}\frac{1}{(1-\overline{w_i}w_j)^\alpha}\Log(1-|w_j|^2)\mathrm{Res}(\frac{1}{w_j-z},w_j)\\
\label{VIij=...}&=&-\frac{1}{\alpha-1}\frac{1}{(1-\overline{w_i}w_j)^\alpha}\Log(1-|w_j|^2).
\end{eqnarray}
By \eqref{IVij=...} and \eqref{VIij=...}, we have
\begin{equation}\label{IVij=...final}
\IV_{ij}=-\frac{1}{(\alpha-1)^2}\frac{1}{(1-\overline{w_i}w_j)^\alpha}-\frac{1}{\alpha-1}\frac{1}{(1-\overline{w_i}w_j)^\alpha}\Log(1-|w_j|^2).
\end{equation}
By \eqref{IIIij=IVij+Vij}, \eqref{IVij=...final} and \eqref{VIij=...}, we get
\begin{equation}\label{IIIij=...}
\III_{ij}=-\frac{1}{(\alpha-1)^2}\frac{1}{(1-\overline{w_i}w_j)^\alpha}+\frac{1}{\alpha-1}\frac{1}{(1-\overline{w_i}w_j)^\alpha}\Log(1-\overline{w_i}w_j).
\end{equation}
Then combining \eqref{I=II+III}, \eqref{eqn: II =...}, \eqref{III=sumIIIij} and \eqref{IIIij=...}, we have
\begin{eqnarray}
\nonumber \I_f(\alpha)&=&\frac{1}{\alpha(\alpha-1)}\sum_{i,j=1}^kc_i\overline{c_j}\frac{1}{(1-\overline{w_i}w_j)^\alpha}\bigg(\overline{\log f^\alpha(w_i)}+\log f^\alpha(w_j)\bigg)\\
\nonumber&&-\frac{1}{(\alpha-1)^2}\sum_{i,j=1}^k\overline{c_i}c_j\frac{1}{(1-\overline{w_i}w_j)^\alpha}\\
\nonumber&&-\frac{1}{\alpha-1}\sum_{i,j=1}^k\overline{c_i}c_j\frac{1}{(1-\overline{w_i}w_j)^\alpha}\Log(1-\overline{w_i}w_j)\\
\nonumber&=&\frac{1}{\alpha(\alpha-1)}\sum_{i,j=1}^kc_i\overline{c_j}\frac{1}{(1-\overline{w_i}w_j)^\alpha}\bigg(\overline{\log f^\alpha(w_i)}+\log f^\alpha(w_j)-\alpha\Log(1-\overline{w_i}w_j)\bigg)\\
\label{I=...}&&-\frac{1}{(\alpha-1)^2}N_f^{2\alpha}(\alpha).
\end{eqnarray}
The last equality is because
\begin{equation}\label{N=...}
N_f^{2\alpha}(\alpha)=\|f^\alpha\|_{A_\alpha^2}^2=\langle f^\alpha, f^\alpha\rangle_{A_\alpha^2}=\sum_{i,j=1}^k\overline{c_i}c_j\frac{1}{(1-\overline{w_i}w_j)^\alpha}.
\end{equation}

Finally, plugging in \eqref{I=...} and \eqref{N=...} into \eqref{eqn: derivative}, we get \eqref{eqn: discrete derivative} and \eqref{eqn: Df=...}. This completes the proof.
\end{proof} 

\begin{rem}\label{rem: f hardy norm equals 1}
In \cite{BBSSZ}, the authors raised the question whether $N_f(\alpha)^{2\alpha}=\|f\|_{A_\alpha^{2\alpha}}^{2\alpha}$ is non-increasing in $\alpha$ given that $\|f\|_{H^2}=1$. Using our method, we can also compute the derivative $\frac{\partial}{\partial\alpha}N_f(\alpha)^{2\alpha}$. In fact, by direct computation, we get
\[
\frac{\partial}{\partial\alpha}N_f(\alpha)^{2\alpha}=(\alpha-1)^{-1}N_f(\alpha)^{2\alpha}+(\alpha-1)I_f(\alpha).
\]
By \eqref{I=...}, if $f\in\OOO^\ast$ and $f^\alpha(z)=\sum_{i=1}^kc_i\frac{1}{(1-\overline{w_i}z)^\alpha}$, we have 
\begin{align*}
I_f(\alpha)=&\frac{1}{\alpha(\alpha-1)}\sum_{i,j=1}^kc_i\overline{c_j}\frac{1}{(1-\overline{w_i}w_j)^\alpha}\bigg(\overline{\log f^\alpha(w_i)}+\log f^\alpha(w_j)-\alpha\Log(1-\overline{w_i}w_j)\bigg)\\
&-\frac{1}{(\alpha-1)^2}N_f^{2\alpha}(\alpha).
\end{align*}
Thus
\begin{equation}\label{eqn: derivative of norm power}
\frac{\partial}{\partial\alpha}N_f(\alpha)^{2\alpha}=\frac{1}{\alpha}\sum_{i,j=1}^kc_i\overline{c_j}\frac{1}{(1-\overline{w_i}w_j)^\alpha}\bigg(\overline{\log f^\alpha(w_i)}+\log f^\alpha(w_j)-\alpha\Log(1-\overline{w_i}w_j)\bigg).
\end{equation}
Using \eqref{eqn: derivative of norm power}, one can easily check that if we drop the condition $\|f\|_{H^2}=1$, then there exists $f$ such that $N_f(\alpha)^{2\alpha}$ is increasing.
\end{rem}

An immediate consequence of Theorem \ref{thm: discrete formula} is the following.

\begin{thm}\label{thm: all positive}
	Suppose $f\in\OOO^\ast$, $\alpha>1$ and $f^\alpha=\sum_{i=1}^kc_i\frac{1}{(1-\overline{w_i}z)^\alpha}$. Suppose further that set of points $\{w_1,\cdots, w_k\}$ belong to a single real line, and that $c_i\geq0, i=1\cdots,k$. Then
	\[
	\frac{\partial}{\partial\alpha}N_f(\alpha)\leq0.
	\]
\end{thm}

\begin{proof}
	The proof simply an application of the Jensen's inequality. Note that under our assumption, 
	\[
	c_i\overline{c_j}\frac{1}{(1-\overline{w_i}w_j)^\alpha}\geq0,\quad f^\alpha(w_j)=\sum_{j=1}^kc_i\frac{1}{(1-\overline{w_i}w_j)^\alpha}\geq0, \quad\forall i,j=1,\cdots,k.
	\]
	Without loss of generality, let us assume that $N_f(\alpha)=1$. That is 
	\[
	\sum_{i,j=1}^kc_i\overline{c_j}\frac{1}{(1-\overline{w_i}w_j)^\alpha}=1.
	\]
	Then by the Jensen's inequality, we have
	\begin{eqnarray}
	\nonumber D_f(\alpha)&=&\sum_{i,j=1}^kc_i\overline{c_j}\frac{1}{(1-\overline{w_i}w_j)^\alpha}\Log\bigg(\overline{f^\alpha(w_i)}f^\alpha(w_j)(1-\overline{w_i}w_j)^\alpha\bigg)\\
	\nonumber&\leq&\Log\bigg(\sum_{i,j=1}^kc_i\overline{c_j}\overline{f^\alpha(w_i)}f^\alpha(w_j)\bigg)\\
	\nonumber&=&\Log N_f^{4\alpha}(\alpha)\\
	\label{eqn: proof of all positive}&=&0.
	\end{eqnarray}
	The second equality is because
	\[
	\sum_{i=1}^kc_i\overline{f^\alpha(w_i)}=\sum_{i=1}^kc_i\sum_{j=1}^k\overline{c_j}\frac{1}{(1-\overline{w_i}w_j)^\alpha}=N_f^{2\alpha}(\alpha).
	\]
	By \eqref{N=...} and \eqref{eqn: proof of all positive}, we have $\frac{\partial}{\partial\alpha}N_f(\alpha)\leq0$. This completes the proof.
\end{proof}

From the proof of Theorem \ref{thm: all positive} we know that the inequality $\frac{\partial}{\partial\alpha}N_f(\alpha)\leq0$ holds true if we can ``formally'' apply the Jensen's inequality. However, in general, the coefficients involved are not positive, and one needs to find other ways.

\section{Sufficient Conditions}\label{sec: sufficient conditions}

In this section, we give some other sufficient conditions for $\frac{\partial}{\partial\alpha}N_f(\alpha)$ to be non-positive. We want to consider the right hand side of \eqref{eqn: Df=...} under a suitable general setting. For this, let us first discuss about how Theorem \ref{thm: discrete formula} can be used to answer Question 3.

In \eqref{eqn: Df=...}, the term $\log f^\alpha(w_i)$ depends on the fact that $f$ is an outer function: the imaginary part of $\log f^{\alpha}(w_i)$ depends on the formula (assuming $f(0)>0$)
\begin{equation}\label{eqn: outer function formula}
\log f(z)=\frac{1}{2\pi}\int_{-\pi}^\pi\frac{e^{i\theta}+z}{e^{i\theta}-z}\Log|f(e^{i\theta})|\mathrm{d}\theta.
\end{equation}
However, it is unclear how this formula could enter the estimates. Things are relatively easy if we are able to apply the single-valued branch $\Log z$ to all $f^\alpha(w_i)$. It turns out that such special cases are enough for our purpose (See Proposition \ref{prop: K epsilon}). Before going into details, let us fix some notations.

\medskip

\noindent\textbf{Notations:}(1) In the rest of this paper, we use $k$ to denote a positive integer. If not otherwise specified, $\mathbf c$ denotes a $k$-tuple of complex numbers, and $\mathbf w$ denotes a $k$-tuple of points in $\DD$. that is, $\mathbf c=(c_1, c_2,\cdots, c_k)$, $c_i\in\CC$, $\mathbf w=(w_1, w_2,\cdots, w_k)$, $w_i\in\DD$. Given $\mathbf c$, $\mathbf w$ and $\alpha>0$, we use $\mathbf W_\alpha$ to denote the $k\times k$ matrix with entry $\frac{1}{(1-\overline{w_i}w_j)^\alpha}$ in the $i$-th row and $j$-th column.  Thinking of $\mathbf c$ as a row vector, we reserve the notation $\mathbf f_\alpha=(f_{1,\alpha},\cdots,f_{k,\alpha})$ for the row vector defined by $\mathbf f_\alpha=\mathbf c\mathbf W_\alpha$. Denote $N_\alpha=\mathbf c \mathbf W_\alpha\mathbf c^*=\sum_{i,j=1}^kc_i\overline{c_j}\frac{1}{(1-\overline{w_i}w_j)^\alpha}$. If $\alpha$ is specified, then we drop the subscription $\alpha$.

\medskip

\noindent(2)  It is well-known that for $\alpha>0$ and $w\in\DD$, the functions 
\[
K_{w,\alpha}(z)=\frac{1}{(1-\overline{w}z)^\alpha}, \quad w\in\DD
\]
define a unique reproducing kernel Hilbert space on $\DD$ \cite{Aro50}. If $\alpha>1$, the space is $A_\alpha^2$; if $\alpha=1$, it is $H^2=A_1^2$. In this paper, for any $\alpha>0$, we use $A_\alpha^2$ to denote the uniquely defined reproducing kernel Hilbert space determined by $\{K_{w,\alpha}: w\in\DD\}$.

\medskip

We find it convenient to consider the following general setting. 
\begin{Def}\label{def: Dalpha}
	(1) Let $\HH:=\{z\in\CC: \mathrm{Re}z>0\}$ denote the right half plane.  Suppose $\alpha>0$, $\mathbf c\in\CC^k$ and $\mathbf w\in\DD^k$ satisfies $f_{i,\alpha}\in\HH, i=1,\cdots,k$, where $\mathbf f_\alpha=(f_{1,\alpha},\cdots,f_{k,\alpha})$ is defined as above.	Define
	\begin{align}
	\label{eqn: Dalpha 1}D_\alpha(\mathbf c, \mathbf w)&=&\sum_{i,j=1}^kc_i\overline{c_j}\frac{1}{(1-\overline{w_i}w_j)^\alpha}\bigg(\Log\overline{f_{i,\alpha}}+\Log f_{j,\alpha}+\alpha\Log(1-\overline{w_i}w_j)\bigg)-N_\alpha\Log N_\alpha\\
	\label{eqn: Dalpha 2}&=&2\mathrm{Re}\sum_{i=1}^kc_i\overline{f_{i,\alpha}\Log f_{i,\alpha}}+\alpha\sum_{i,j=1}^kc_i\overline{c_j}\frac{1}{(1-\overline{w_i}w_j)^\alpha}\Log(1-\overline{w_i}w_j)-N_\alpha\Log N_\alpha.
	\end{align}
	
	\medskip

		\noindent(2) For any $\alpha>0$ and $0<\varepsilon\leq1$, define
		\[
		\Lambda_{\alpha, \varepsilon}=\bigg\{(\mathbf c, \mathbf w): 
		\mathbf c\in\CC^k, \mathbf w\in(-\varepsilon, \varepsilon)^k,
		f_{i,\alpha}\in\HH, i=1,\cdots,k,k \text{ is a positive integer. } \bigg\}
		\]
		and
		\[
		\KKK_{\alpha, \varepsilon}=\bigg\{\sum_{i=1}^kc_i\frac{1}{(1-\overline{w_i}z)^\alpha}: (\mathbf c, \mathbf w)\in \Lambda_{\alpha, \varepsilon}\bigg\}.
		\]
	
		We use $\Lambda_\alpha$, $\KKK_\alpha$ to denote $\Lambda_{\alpha,1}$, $\KKK_{\alpha,1}$.
		
		\medskip
		
	\noindent(3) Define
	\[
	\Gamma = \bigg\{(\mathbf c, \mathbf w): \mathbf c\in\RR^k, \mathbf w\in(-1, 1)^k, k \text{ is a positive integer}\bigg\}.
	\]
	For $(\mathbf c, \mathbf w)\in \Gamma$ and $\alpha>0$, define
	\begin{equation}\label{eqn: Dalpha 3}
	D_\alpha(\mathbf c, \mathbf w)=2\sum_{i=1}^kc_if_{i,\alpha}\Log|f_{i,\alpha}|+\sum_{i,j=1}^kc_ic_j\frac{1}{(1-w_iw_j)^\alpha}\Log(1-w_iw_j)^\alpha-N_\alpha\Log N_\alpha.
	\end{equation}
	Note that since $x\Log x$ tends to $0$ as $x$ tends to $0$, the definition above makes sense even if $f_i=0$ for some $i=1,\cdots,k$. It is also easy to see that \eqref{eqn: Dalpha 3} coincides with \eqref{eqn: Dalpha 1} when $f_{i,\alpha}>0, \forall i=1,\cdots,k$.
\end{Def}

\begin{rem}\label{rem 1}
Suppose $(\mathbf c, \mathbf w)\in \Gamma$ and $f^\alpha(z)=\sum_{i=1}^kc_iK_{w,\alpha}(z)\in\OOO^\ast$, then by \eqref{eqn: Df=...} and \eqref{eqn: Dalpha 3}, it is easy to see that $D_f(\alpha)=D_\alpha(\mathbf c,\mathbf w)$. If $(\mathbf c, \mathbf w)\in \Lambda_\alpha$ and $f^\alpha=\sum_{i=0}^kc_iK_{w,\alpha}\in\OOO^\ast$, then it is not necessarily true that $D_f(\alpha)=D_\alpha(\mathbf c, \mathbf w)$. However, if one knows that $\{w_1,\cdots,w_k\}$ is contained in a connected open subset $\Omega$ of $\DD$ which is mapped, by $f^\alpha$, into $\HH$, then by standard argument, the function $\Log f^\alpha(z)|_\Omega$ differs from the function given in \eqref{eqn: outer function formula}, by an integer multiple of $2\pi i$. Then from the expression of \eqref{eqn: Df=...} one can see that $D_f(\alpha)=D_\alpha(\mathbf c, \mathbf w)$. We will use this fact later.
\end{rem}

It turns out that we only need to consider the case when $f^\alpha\in\KKK_{\alpha, \varepsilon}$ for $\varepsilon$ small enough.
\begin{prop}\label{prop: K epsilon}
	Suppose for any $\alpha>1$ there exists $0<\varepsilon\leq1$ such that $D_\alpha(\mathbf c, \mathbf w)\leq0$ for all $(\mathbf c, \mathbf w)\in \Lambda_{\alpha, \varepsilon}$. Then $\frac{\partial}{\partial\alpha}N_f(\alpha)\leq0$ for all $f\in\OOO^\ast$. As a consequence, Conjecture 2 holds.
\end{prop}

The proof is based on the following two lemmas.

\begin{lem}\label{lem: K is dense in O}
	Suppose $\alpha>1$ and $0<\varepsilon\leq1$. Then for any $g\in\OOO^\ast$ such that $g(0)=1$, there exists $0<\delta\leq\varepsilon$ and a sequence $\{g_n\}\in\KKK_{\alpha, \delta}$ such that $g_n$ converges uniformly on $\overline{\DD}$ to $g$. Moreover, $g_n(z)\in\HH$ for all $n$ and all $z\in\DD$ with $|z|<\delta$.
\end{lem}

\begin{proof}
	Since $g(0)=1$, we can choose $0<\delta\leq\varepsilon$ such that $\mathrm{Re}g(z)>\frac{1}{2}$ for any $z\in\DD$ with $|z|<\delta$.
	Choose $r>1$ such that $g$ is defined on $\{z\in\CC: |z|\leq r\}$. Define
	\[
	g_r(z)=g(rz),\quad z\in\DD.
	\]
	Obviously, $g_r\in A_\alpha^2$. The subspace
	\[
	\mathrm{span}\{K_{w,\alpha}: w\in(-\delta, \delta)\}
	\]
	is dense in $A_\alpha^2$. Choose a sequence $\{\tilde{g}_n\}\subset\mathrm{span}\{K_{w,\alpha}: w\in(-\delta, \delta)\}$ such that $\tilde{g}_n\to g_r$ in $A_\alpha^2$ norm. Then $\tilde{g}_n$ converge uniformly to $g_r$ on $\{z\in\CC: |z|\leq \frac{1}{r}\}$. Define
	\[
	g_n(z)=\tilde{g}_n(\frac{z}{r}),\quad n=1,2,\cdots.
	\]
	Then $g_n$ converge uniformly to $g$ on $\overline{\DD}$. By construction, each $\tilde{g}_n$ is of form
	\[
	\tilde{g}_n=\sum_{i=1}^{k_n}c_{n,i}K_{w_{n,i},\alpha},
	\]
	where $w_{n,i}\in(-\delta, \delta)$, $\forall i$. Therefore
	\[
	g_n(z)=\sum_{i=1}^{k_n}c_{n,i}\frac{1}{(1-\frac{w_{n,i}}{r}z)^\alpha}=\sum_{i=1}^{k_n}c_{n,i}K_{\frac{w_{n,i}}{r},\alpha}.
	\]
	Since $r>1$ we have $\frac{w_{n,i}}{r}\in(-\delta, \delta)$. Also, since $g_n$ converge uniformly to $g$ and $\mathrm{Re}g(z)>\frac{1}{2}$ if $|z|<\delta$, by passing to a subsequence, we have $g_n(z)\in\HH$ for any $n$ and any $z$ with $|z|<\delta$. In particular, $g_n(\frac{w_{n,i}}{r})\in\HH$. Therefore $g_n\in\KKK_{\alpha, \delta}$ for any $n$. This completes the proof.
\end{proof}

The following lemma is simply a consequence of the fact that, for $f\in H^2$, $f_r(z):=f(rz)$ converges to $f$ in $H^2$ norm as $r\to1-$.
\begin{lem}\label{lem: O is dense in outer}
	Suppose $f$ is an outer function in $H^2$. Then there exists a sequence $\{f_n\}\subset\OOO^\ast$ such that $f_n$ tends to $f$ in the Hardy norm $\|\cdot\|_{H^2}$.
\end{lem}

\begin{proof}[\textbf{Proof of Proposition \ref{prop: K epsilon}}]
	Assume that for some $\alpha>1$ and $0<\varepsilon\leq1$ we have $D_\alpha(\mathbf c, \mathbf w)\leq0$ for all $(\mathbf c, \mathbf w)\in \Lambda_{\alpha, \varepsilon}$. For any $f\in\OOO^\ast$, we want to show $\frac{\partial}{\partial\alpha}N_f(\alpha)\leq0$. Without loss of generality we can assume $f(0)=1$. By Lemma \ref{lem: K is dense in O}, there exists $0<\delta\leq\varepsilon$ and a sequence $\{g_n\}$ in $\KKK_{\alpha, \delta}$ such that $g_n$ converges uniformly to $f^\alpha$ on $\DD$ and $g_n$ maps $\{z\in\DD: |z|<\delta\}$ to $\HH$. Also, since $|f^\alpha|$ is bounded away from $0$ on $\DD$, for $n$ large enough, $g_n$ is outer and we can define $f_n=g_n^{1/\alpha}$.
	By Remark \ref{rem 1}, we have $\frac{\partial}{\partial\alpha}N_{f_n}(\alpha)=D_\alpha(\mathbf c_n, \mathbf w_n)\in \Lambda_{\alpha,\delta}$, where $g_n$ corresponds to $(\mathbf{c_n}, \mathbf{w_n})$. In particular, we have $\frac{\partial}{\partial\alpha}N_{f_n}(\alpha)\leq0$ for $n$ large enough.
	
	On the other hand, since $|f|$ is bounded away from $0$ we also have $\Log |f_n|$ converging uniformly to $\Log|f|$ on $\DD$. By \eqref{eqn: derivative} and \eqref{eqn: derivative I} it is easy to see that $\frac{\partial}{\partial\alpha}N_{f_n}(\alpha)\to\frac{\partial}{\partial\alpha}N_f(\alpha)$ as $n$ tend to infinity. Thus $\frac{\partial}{\partial\alpha}N_f(\alpha)\leq0$. 
	Then by Lemma \ref{lem: limit is hardy norm}, we have $\|f\|_{A_\alpha}^{2\alpha}\leq\|f\|_{H^2}$ for all $f\in\OOO^\ast$. By Lemma \ref{lem: O is dense in outer}, the inequality also holds for all outer functions. Suppose $f=\eta g$ where $\eta$ is inner and $g$ is outer. Then $\|f\|_{A_\alpha^{2\alpha}}\leq\|g\|_{A_\alpha^{2\alpha}}\leq\|g\|_{H^2}=\|f\|_{H^2}$. This completes the proof.
\end{proof}

Now we are ready to give some sufficient conditions for $D_\alpha(\mathbf c,\mathbf w)$ to be non-positive.

\begin{thm}\label{thm: all except one positive}
	Suppose $\alpha>0$ and $(\mathbf c, \mathbf w)\in \Lambda_{\alpha}\cup \Gamma$ satisfy the following conditions.
	\begin{itemize}
		\item[(1)] $w_1<w_2<\cdots<w_k$, where $k$ is the number of entries in $\mathbf w$;
		\item[(2)] either $\{c_2,\cdots,c_k\}$ or $\{c_1,\cdots,c_{k-1}\}$ are real and have the same sign.
	\end{itemize}
	Then we have $D_\alpha(\mathbf c, \mathbf w)\leq0$.
\end{thm}

We will need the following lemma in the proof of Theorem \ref{thm: all except one positive}.

\begin{lem}\label{lem: positive xi}
	Suppose $\alpha>0$, $\mathbf c\in\CC^k$ and $\mathbf A=[a_{ij}]$ is semi-positive definite, $a_{ij}>0$. Let  $\mathbf f=\mathbf c\mathbf A$ and $N=\mathbf c\mathbf A\mathbf c^\ast$. Then for any $x_1, \cdots, x_k\geq0$ we have 
	\begin{equation}\label{eqn: positive xi}
	\sum_{i,j=1}^kx_ix_ja_{ij}\Log\frac{|f_if_j|}{a_{ij}N}\leq0.
	\end{equation}
\end{lem}

\begin{proof}
	The proof is, again, an application of the Jenson's Inequality. If some $f_i$ equals zero then the left hand side is $-\infty$ and the inequality always holds. Assume $f_i$ are all non-zero. Without loss of generality, we can also assume that $\sum_{i,j=1}^kx_ix_ja_{ij}=1$. Applying the Jenson's Inequality, we get
	\begin{equation}\label{eqn: positive xi 1}
		\sum_{i,j=1}^kx_ix_ja_{ij}\Log\frac{|f_if_j|}{a_{ij}N}\leq
	\Log\bigg(\sum_{i,j=1}^kx_ix_ja_{ij}\frac{|f_if_j|}{a_{ij}N}\bigg)\leq\Log\frac{\big(\sum_ix_i|f_i|\big)^2}{N}.
	\end{equation}
	Choose $e_i\in\CC$ such that $|e_i|=1$ and $e_if_i=|f_i|$. Then the right hand side of \eqref{eqn: positive xi 1} becomes $\Log\frac{\big(\sum_{i,j=1}^kx_ie_ic_ja_{ij}\big)^2}{\sum_{i,j=1}^kc_i\overline{c_j}a_{ij}}$, which is less than or equal to $\Log\sum_{i,j=1}^kx_ie_ix_j\overline{e_j}a_{ij}$ by the fact that $A$ is semi-positive definite. Since $x_i\geq0$ and $a_{ij}>0$, we have $\sum_{i,j=1}^kx_ie_ix_j\overline{e_j}a_{ij}\leq\sum_{i,j=1}^kx_ix_ja_{ij}=1$. Therefore the left hand side of \eqref{eqn: positive xi} is less than or equal to $0$. This completes the proof.
\end{proof}

\begin{proof}[\textbf{Proof of Theorem \ref{thm: all except one positive}}]
	We will prove the theorem in the case when $(\mathbf c,\mathbf w)\in \Lambda_\alpha$. The proof when $(\mathbf c,\mathbf w)\in \Gamma$ is similar.
	First, we notice that if we let $-\mathbf w=(-w_1,\cdots,-w_k)$, then $D_\alpha(\mathbf c,\mathbf w)=D_\alpha(\mathbf c, \mathbf{-w})$. From this, it is easy to see that it suffices to consider the case when $w_1<w_2<\cdots<w_k$ and $\{c_2,\cdots,c_k\}$ are real and have the same sign.

	Let us further reduce the cases.
	Suppose $w_1<w_2<\cdots<w_k$. Let $z_i=-\varphi_{w_1}(w_i)$, $i=1,\cdots,k$. Then it is easy to check that $0=z_1<z_2<\cdots<z_k$. Suppose $(\mathbf c, \mathbf w)\in \Lambda_{\alpha}$. Let $d_i=\frac{c_i(1-z_iw_1)^\alpha}{(1-w_1^2)^{\alpha/2}}$ and consider the pair $(\mathbf d,\mathbf z)$. Write $\mathbf W=[\frac{1}{(1-w_iw_j)^\alpha}]$, $\mathbf\Lambda=[\frac{1}{(1-z_iz_j)^\alpha}]$, $\mathbf f=\mathbf c\mathbf W$ and $\mathbf g=\mathbf d\mathbf\Lambda$. Using the well-known formula 
	\[
	\frac{1}{1-\varphi_a(z)\overline{\varphi_a(w)}}=\frac{(1-z\overline{a})(1-a\overline{w})}{(1-|a|^2)(1-z\overline{w})}, \quad z, w, a\in\DD,
	\]
	 it is easy to check that
	\[
	\mathbf g=\mathbf f\mathrm{diag}\bigg(\frac{(1-w_1^2)^{\alpha/2}}{(1-z_iw_1)^\alpha}\bigg).
	\]
	Then $(\mathbf d, \mathbf z)\in \Lambda_{\alpha}$. From the equation above, it is also straight-forward to check that $D_\alpha(\mathbf c, \mathbf w)=D_\alpha(\mathbf d,\mathbf z)$. Therefore, in order to prove Theorem \ref{thm: all except one positive}, we only need to consider the case when $0=w_1<\cdots<w_k$ and $\{c_2,\cdots,c_k\}$ are real and have the same sign.
	
	Assume that $(\mathbf c, \mathbf w)\in \Lambda_{\alpha}$, $0=w_1<\cdots<w_k$ and $c_i\geq0, \forall i=2,\cdots,k$.  The case when $c_2=\cdots=c_k=0$ is trivial. Thus we can assume that $c_i>0$ for some $i=2,\cdots,k$.	Define $\mathbf f$ and $\mathbf W$ as before. The idea is to find a non-increasing function that takes value $D_\alpha(\mathbf c,\mathbf w)$ at $\alpha$ and $0$ at $0$.
	
	For $0\leq t\leq\alpha$, define $a_{ij, t}=\frac{1}{(1-w_iw_j)^t}$, $i,j=1,\cdots,k$. Define $\mathbf W_t=[a_{ij,t}]$, $N_t=\mathbf c\mathbf W_t\mathbf c^*$ and $\mathbf f_t=\mathbf c \mathbf W_t$. Since only $c_1$ may have imaginary part, the signs of the imaginary part of each $f_{i,t}$ depend only on that of $c_1$. Assume, without loss of generality, that $\mathrm{Im}c_1\geq0$. Then $\mathrm{Im}f_{i,t}\geq0, \forall i=1,\cdots,k$. Define
	\[
	D_t=2\mathrm{Re}\sum_{i=1}^kc_i\overline{f_{i,t}}\overline{\Log f_{i,t}}-\sum_{i,j=1}^kc_i\overline{c_j}a_{ij,t}\Log a_{ij,t}-N_t\Log N_t,\quad 0\leq t\leq\alpha.
	\]
	Notice that since $w_1=0$, we have $\frac{\mathrm{d}}{\mathrm{d}t}f_{i,t}=\sum_{j=2}^kc_j\frac{1}{(1-w_iw_j)^t}\Log\frac{1}{1-w_iw_j}>0, \forall i=1,\cdots,k$. So the points $t$ such that $f_{i,t}=0$ for some $i$, are isolated. Also, $z\Log z\to0$ if $z$ tends to $0$. From this we can see that $D_t$ is a continuous, piecewise differentiable function. 
	
	Next, we show that $D_t$ is non-increasing. By the previous argument, it suffices to show that $\frac{\mathrm{d}}{\mathrm{d}t}D_t\leq0$ at the points where each $f_{i,t}$ is non-zero. By direct computation, we get
	\[
	\frac{\mathrm{d}}{\mathrm{d}t}D_t=\sum_{i,j=1}^kc_i\overline{c_j}\bigg(\frac{\mathrm{d}}{\mathrm{d}t}a_{ij,t}\bigg)\bigg(\overline{\Log f_{i,t}}+\Log f_{j,t}-\Log(a_{ij,t}N_t)\bigg).
	\]
	 Since $w_1=0$ and $c_i\geq0, \forall i=2,\cdots,k$, we have
	\begin{eqnarray*}
	\frac{\mathrm{d}}{\mathrm{d}t}D_t
	&=&\sum_{i,j=2}^kc_ic_j\frac{1}{(1-w_iw_j)^t}\Log\big(\frac{1}{1-w_iw_j}\big)\bigg(\overline{\Log f_{i,t}}+\Log f_{j,t}-\Log(a_{ij,t}N_t)\bigg)\\
	&=&\sum_{i,j=2}^kc_ic_ja_{ij,t}\big(\sum_{n=1}^\infty\frac{1}{n}w_i^nw_j^n\big)\Log\frac{|f_{i,t}f_{j,t}|}{a_{ij,t}N_t}\\
	&=&\sum_{n=1}^\infty\frac{1}{n}\sum_{i,j=2}^kc_iw_i^nc_jw_j^na_{ij,t}\Log\frac{|f_{i,t}f_{j,t}|}{a_{ij,t}N_t}\\
	&\leq&0.
	\end{eqnarray*}
Here the last inequality is by Lemma \ref{lem: positive xi}. If $\{c_2,\cdots,c_k\}$ are all non-positive, simply replace $c_i$ with $-c_i$ in the above argument. Thus in either case we have that $D_t$ is non-increasing.

It is obvious that $D_\alpha=D_\alpha(\mathbf c,\mathbf w)$. By straight-forward computation it is also easy to show that $D_0=0$. Therefore $D_\alpha(\mathbf c,\mathbf w)=D_\alpha\leq D_0=0$. This completes the proof.
\end{proof}

\begin{thm}\label{thm: two points}
	Suppose $\alpha>0$, $\mathbf c\in\CC^2$, $\mathbf w\in\DD^2$ and $(\mathbf c, \mathbf w)\in\Lambda_\alpha\cup\Gamma$. 
	Then
	\[
	D_\alpha(\mathbf c, \mathbf w)\leq0.
	\]
	Moreover, $D_\alpha(\mathbf c,\mathbf w)=0$ if and only if $c_1=0$, or $c_2=0$, or $w_1=w_2$.
\end{thm}

\begin{proof}
	The proof is similar as that of Theorem \ref{thm: all except one positive}. Define $\mathbf W=[a_{ij}]=[\frac{1}{(1-w_iw_j)^\alpha}]$ as before.
	For $0\leq t\leq1$, define
	\[
	a_{ij,t}=\begin{cases}
	(1-t)a_{11}+t\frac{a_{12}^2}{a_{22}}, &i=j=1\\
	a_{ij}, &\text{otherwise.}
	\end{cases}
	\]
	Write $\mathbf W_t=[a_{ij,t}]$, $N_t=\mathbf c\mathbf W_t\mathbf c^*$ and $\mathbf f_t=\mathbf c\mathbf W_t$.
	It is easy to check that the following hold.
	\begin{itemize}
		\item[(i)] $f_{1,t}=(1-t)f_1+t\frac{a_{12}}{a_{22}}f_2$, $f_{2,t}\equiv f_2$;
		\item[(ii)] $\frac{\mathrm{d}}{\mathrm{d}t}a_{11,t}\leq0$;
		\item[(iii)] each $\mathbf W_t$ is semi-positive definite;
		\item[(iv)] $\mathbf W_1$ has rank $1$.
	\end{itemize}
	 By (i), the paths $f_{1,t}$ and $f_{2,t}$ stay in $\HH$. Define
	\[
	D_t:=2\mathrm{Re}\sum_{i=1}^2c_i\overline{f_{i,t}}\Log\overline{f_{i,t}}-\sum_{i,j=1}^2c_i\overline{c_j}a_{ij,t}\Log a_{ij,t}-N_t\Log N_t,
	\]
	Then $D_t$ is a differentiable function on $(0,1)$.
	From (iv) it is easy to compute that $D_1=0$. 
	By direct computation, we have
	\[
	\frac{\mathrm{d}}{\mathrm{d}t}D_t=|c_1|^2\bigg(\frac{\mathrm{d}}{\mathrm{d}t}a_{11,t}\bigg)\Log\frac{|f_{1,t}|^2}{a_{11,t}N_t}.
	\]
	Since $\mathbf W_t$ is positive definite, by Lemma \ref{lem: positive xi}, it is easy to see that $\frac{\mathrm{d}}{\mathrm{d}t}D_t\geq0$.
	Thus $D_\alpha(\mathbf c,\mathbf w)=D_0\leq D_1=0$. Equality holds if and only if $\frac{\mathrm{d}}{\mathrm{d}t}D_t\equiv0$. This always holds when $w_1= w_2$. If $w_1\neq w_2$, then $\frac{\mathrm{d}}{\mathrm{d}t}a_{11,t}\neq0$. Thus either $c_1=0$ or $\frac{|f_{1,t}|^2}{a_{11,t}N_t}\equiv1$. In particular, if $c_1\neq0$, then 
	\[
	|f_{1,0}|^2=|\mathbf c\mathbf W_0\mathbf e_1^*|^2=a_{11,0}N_0=(\mathbf e_1\mathbf W_0\mathbf e_1^*)(\mathbf c\mathbf W_0\mathbf c^*).
	\]
	Here $\mathbf e_1=(1,0,\cdots,0)$. Since $W_0$ is positive definite, this occurs only when $c_2=0$. This completes the proof.
\end{proof}

In terms of $\frac{\partial}{\partial\alpha}N_f(\alpha)$, we summarize our results as follows.
\begin{thm}\label{thm: all results about derivative}
	Suppose $f\in\OOO^\ast$, $\alpha>1$, $f^\alpha=\sum_{i=1}^k c_iK_{w_i,\alpha}$ and $(\mathbf c,\mathbf w)\in \Lambda_\alpha\cup \Gamma$. Suppose one of the following holds.
	\begin{itemize}
		\item[(1)] $c_i\geq0$, $i=1,\cdots,k$.
		\item[(2)] $w_1<\cdots<w_k$, and either $\{c_2,\cdots,c_k\}$ or $\{c_1,\cdots,c_{k-1}\}$ are real and have the same sign.
		\item[(3)] $k=2$.
	\end{itemize}
Then we have $\frac{\partial}{\partial\alpha}N_f(\alpha)\leq0$.
\end{thm}

\section{Norm Inequalities for Linear Combinations of Two Reproducing Kernels}\label{sec: an application}

Recall that in Lemma \ref{lem: limit is hardy norm}, we showed that if $\frac{\partial}{\partial\alpha}N_f(\alpha)\leq0, \forall\alpha>1$, for some $f\in\OOO^\ast$, then Conjecture 2 holds for $f$. In this section, we provide an alternative way of proving results on Conjecture 2, using results obtained in Section \ref{sec: sufficient conditions}. As a consequence, we prove the following theorem.

\begin{thm}\label{thm: application}
	Suppose $f\in H^2$ and $f=\eta F$, where $\eta$ is inner and $F$ has no zeros in $\DD$. Suppose
		\[
		F^\alpha=c_1K_{w_1,\alpha}+c_2K_{w_2,\alpha}.
		\] 
	for some $\alpha>1$, and $\mathbf c\in\CC^2, \mathbf w\in\DD^2$.
Then for any $1\leq\beta\leq\alpha$, we have 
\begin{equation}\label{eqn: norm ineq 1}
\|F\|_{A_\alpha^{2\alpha}}\leq\|F\|_{A_\beta^{2\beta}}.
\end{equation}
Equality holds if and only if $F^\alpha=cK_{w,\alpha}$ for some $c\in\CC$ and $w\in\DD$.
As a consequence, we have
\begin{equation}\label{eqn: norm ineq 2}
\|f\|_{A_\alpha^{2\alpha}}\leq\|f\|_{H^2}.
\end{equation}
Equality holds if and only if $f=cK_{w,1}$ for some $c\in\CC$ and $w\in\DD$.
\end{thm}

The proof is based on a different way of viewing $D_\alpha(\mathbf c,\mathbf w)$. Recall that in Definition \ref{def: Dalpha}, for $\alpha>0$ and $(\mathbf c,\mathbf w)\in\Lambda_\alpha\cup\Gamma$, we defined $W_\alpha=[\frac{1}{(1-\overline{w_i}w_j)^\alpha}]$ and $\mathbf f_\alpha=\mathbf c\mathbf W_\alpha$. Then $D_\alpha(\mathbf c,\mathbf w)$ is defined using $\mathbf f_\alpha$ and $\mathbf W_\alpha$. In the case when $\{w_i: i=1,\cdots,k\}$ are distinct points, the matrix $\mathbf W_\alpha$ is invertible. Therefore we have $\mathbf c=\mathbf f_\alpha\mathbf W_\alpha^{-1}$. This means we can define $D_\alpha(\mathbf c,\mathbf w)$ using $\mathbf f_\alpha$.

\begin{Def}\label{def: Dhatalpha}
	Suppose $\alpha>0$, $k$ is a positive integer, and $\mathbf w\in\DD^k$ is such that $\{w_i: i=1,\cdots,k\}$ are distinct. Define $\mathbf W_\alpha$ as usual. Suppose either $\mathbf f\in\HH^k$ or $\mathbf w\in(-1,1)^k, \mathbf f\in\RR^k$. Let $\mathbf c_\alpha=\mathbf f\mathbf W_\alpha^{-1}$. Define
	\begin{equation}\label{eqn: Dhatalpha1}
	\widehat{D}_\alpha(\mathbf f,\mathbf w)=D_\alpha(\mathbf c_\alpha,\mathbf w).
	\end{equation}
\end{Def}

\begin{Def}
	Suppose $\mathbf w\in\DD^k$, $\{w_1,\cdots,w_k\}$ are distinct, and $\alpha>0$. Define
	\[
	\KKK_{\mathbf w,\alpha}=\mathrm{span}\{K_{w_i,\alpha}: i=1,\cdots,k\}\subset A_\alpha^{2}
	\]
	and $P_{\mathbf w,\alpha}$ the orthogonal projection from $A_\alpha^2$ onto $\KKK_{\mathbf w,\alpha}$. For $f\in A_\alpha^2$, if we denote $f(\mathbf w)=(f(w_1),\cdots,f(w_k))$, then it is easy to compute that 
	\[
	\|P_{\mathbf w,\alpha}(f)\|_{A_\alpha^2}^2=f(\mathbf w)\mathbf W_\alpha^{-1}f(\mathbf w)^*.
	\]
\end{Def}

\begin{proof}[\textbf{Proof of Theorem \ref{thm: application}}]
	It is easy to see that \eqref{eqn: norm ineq 2} follows from \eqref{eqn: norm ineq 1}. Thus we only need to prove \eqref{eqn: norm ineq 1}.
	
\noindent(1) First, we prove \eqref{eqn: norm ineq 1} under the following conditions.
\begin{itemize}
	\item[(i)] $F^\alpha=c_1+c_2K_{w,\alpha}$, $w\geq0$;
	\item[(ii)] There is a connected open neighborhood $\Omega$, of $\{0, w\}$, such that $F^\alpha(\Omega)\subset\HH$.
\end{itemize}
Assume the above, then we can choose $\log F^\alpha$ such that $\log F^\alpha|_\Omega=\Log F^\alpha|_\Omega$. As a consequence, $F^\beta|_\Omega\subset\HH$ for any $\beta\in[1,\alpha]$. 

For $1\leq\beta\leq\alpha$, consider 
\[
N_\beta:=\bigg(F^\beta(\mathbf w)\mathbf W_\beta^{-1}F^\beta(\mathbf w)^*\bigg)^{1/\beta}=\|P_{\mathbf w,\beta}F^\beta\|_{A_\beta^2}^{2/\beta},
\]
where $F^\beta(\mathbf w)=(F^\beta(w_1), F^\beta(w_2))$.
Then by straight-forward computation, we have
\[
\frac{\mathrm{d}}{\mathrm{d}\beta}N_\beta=\frac{1}{\beta^2}N^{1-\beta}_\beta\widehat{D}_\beta(F^\beta(\mathbf w), \mathbf w)\leq0.
\]
The last inequality is because of Theorem \ref{thm: two points}. Therefore we have
\[
\|F\|_{A_\alpha^{2\alpha}}=\|F^\alpha\|_{A_\alpha^2}^{1/\alpha}=N_\alpha^{1/2}\leq N_\beta^{1/2}=\|P_{\mathbf w,\beta}F^\beta\|_{A_\beta^2}^{1/\beta}\leq\|F^\beta\|_{A_\beta^2}^{1/\beta}=\|F\|_{A_\beta^{2\beta}}.
\]
If we assume that $w\neq0$, then by Theorem \ref{thm: two points}, we also know that the equality holds if and only if for each $\beta$, either $c_{1,\beta}=0$ or $c_{2,\beta}=0$. In particular, either $c_1=0$ or $c_2=0$. On the other hand, if either $c_1=0$ or $c_2=0$, then it is easy to check that the equality in \eqref{eqn: norm ineq 1} holds.
This completes the proof for case (1).

\medskip

\noindent(2) Next, we consider the case when $F^\alpha=c_1+c_2K_{w,\alpha}, w\in\DD, c_1, c_2\in\CC$. Choose $\theta\in[0,2\pi]$ so that $e^{i\theta}w\geq0$. Let $F_\theta(z)=F(e^{-i\theta}z)$. Then $F^\alpha_\theta=c_1+c_2K_{e^{i\theta}w,\alpha}$. Inequality \eqref{eqn: norm ineq 1} for $F$ follows from \eqref{eqn: norm ineq 1} for $F_\theta$. Thus we may assume that $w\geq0$ in the beginning. Suppose $F$ has no zeros in $\DD$ and $F^\alpha=c_1+c_2K_{w,\alpha}$ with $w\geq0$. We will show that after multiplying $F$ by a non-zero constant, the condition (ii) in case (1) will be satisfied. This will lead to \eqref{eqn: norm ineq 1} for case (2). We may as well assume that $F^\alpha=1+cK_{w,\alpha}$ with $w\geq0$. It is easy to see that $F^\alpha$ maps the interval $[0,w]$ onto the (complex valued) interval between $1+c$ and $1+\frac{c}{(1-|w|^2)^\alpha}$, which we denote by $I$. Since $F$ has no zeros in $\DD$, $0\notin I$. Thus $I$ must be contained in some half plane $e^{i\theta_1}\HH$. By standard trick we can find a connected open neighborhood $\Omega$ of $[0,w]$ such that $F^\alpha(\Omega)\subset e^{i\theta_1}\HH$. Therefore $\big(e^{-i\theta_1/\alpha}F\big)^\alpha(\Omega)\subset\HH$. This completes the proof for case (2).

\medskip

\noindent(3) In general, suppose $F$ has no zeros in $\DD$ and $F^\alpha=c_1K_{w_1,\alpha}+c_2K_{w_2,\alpha}$. Denote
\[
k_{w,\beta}(z)=\frac{(1-|w|^2)^{\beta/2}}{(1-\overline{w}z)^\beta}, \quad w\in\DD, \beta>0.
\]
Let $G=F\circ\varphi_{w_1}\cdot k_{w_1,1}$. Then $G$ also has no zeros in $\DD$ and $G^\beta=F^\beta\circ\varphi_{w_1}\cdot k_{w_1,\beta}$. It is standard to check that $\|G^\beta\|_{A_\beta^2}=\|F^\beta\|_{A_\beta^2}$, $\forall\beta\in[1,\alpha]$. Notice that \eqref{eqn: norm ineq 1} is equivalent to 
\[
\|F^\alpha\|_{A_\alpha^2}^{1/\alpha}\leq\|F^\beta\|_{A_\beta^2}^{1/\beta}, \quad\beta\in[1,\alpha].
\]
So it suffices to prove 
\[
\|G^\alpha\|_{A_\alpha^2}^{1/\alpha}\leq\|G^\beta\|_{A_\beta^2}^{1/\beta}, \quad \beta\in[1,\alpha],
\]
which is in turn, equivalent to $\|G\|_{A_\alpha^{2\alpha}}\leq\|G\|_{A_\beta^{2\beta}}, \beta\in[1,\alpha]$.
By straight-forward computation we have
\[
G^\alpha=\frac{c_1}{(1-|w_1|^2)^{\alpha/2}}+c_2\frac{(1-|w_1|^2)^{\alpha/2}}{(1-\overline{w_2}w_1)^\alpha}K_{\varphi_{w_1}(w_2),\alpha}.
\]
Thus $G$ satisfies case (2). This completes the proof of \eqref{eqn: norm ineq 1} in the general case. Tracing back to the proof of (1), we also see that the equality in \eqref{eqn: norm ineq 1} holds if and only if $F^\alpha=cK_{w,\alpha}$ for some $c\in\CC$ and $w\in\DD$. Then \eqref{eqn: norm ineq 2} follows immediately. This completes the proof.
\end{proof}

In terms of Conjecture 1, \eqref{eqn: norm ineq 2} becomes the following.
\begin{cor}\label{cor: conj 1 for two points}
Suppose $0<p\leq2$ and $\alpha=\frac{2}{p}$. Suppose $f\in H^p$, $f=\eta F$, where $\eta$ is inner and $F$ has no zeros in $\DD$. If 
\[
F=c_1K_{w_1,\alpha}+c_2K_{w_2,\alpha},
\]
then 
\[
\|f\|_{A_{2/p}^2}\leq\|f\|_{H^p}.
\]
Equality holds if and only if $f=cK_{w,\alpha}$ for some $c\in\CC$ and $w\in\DD$.
\end{cor}

\section{Remarks and Numerical Evidences}\label{sec: some evidences}
\subsection{Some Further Remarks}
\noindent(1) In Proposition \ref{prop: K epsilon}, we give a sufficient condition for Conjecture 2 to hold. In its most general form, we list the conjecture below.

\medskip

\noindent\textbf{Conjecture 4.} Suppose $\alpha>0$ and $(\mathbf c, \mathbf w)\in \Lambda_\alpha$. Then $D_\alpha(\mathbf c, \mathbf w)$ defined as in \eqref{eqn: Dalpha 1} and \eqref{eqn: Dalpha 3}
is non-positive.

\medskip

Another interesting question to ask is the following.

\medskip

\noindent\textbf{Question 5.} Suppose Conjecture 4 holds. Does it imply that the norms $\|f^\alpha\|^{1/\alpha}_{A_\alpha^2}$ are non-increasing for $\alpha>0$, for some set of functions $f$? 

\medskip

\noindent(2) An immediate observation from \eqref{eqn: Dalpha 1} is that $D_\alpha(\mathbf c, \mathbf w)$ is the bi-linear form $\mathbf c\mathbf W_\alpha\circ \mathbf B_{\mathbf c,\mathbf w}\mathbf c^*$, where $\mathbf B_{\mathbf c,\mathbf w}=\bigg[\overline{\Log f_i}+\Log f_j-\Log (a_{ij,\alpha}N_\alpha)\bigg]_{i,j=1,\cdots,k}$ and $A\circ B$ denotes the Hadmadard product of $A$ and $B$. Thus a sufficient condition of Conjecture 4 would be that $-\mathbf B_{\mathbf c,\mathbf w}$ is semi-positive definite. However, this is not true, even in the simplest case.  Take, for example, $c_1=c_2=1$, $w_1=0$, and $w_2\in(0,1)$ such that $\frac{1}{(1-w_2^2)^\alpha}=1.1$.

\medskip

\noindent(3) In Section \ref{sec: an application}, we also defined the function $\widehat{D}_\alpha(\mathbf f, \mathbf w)$. Then Conjecture 4 is equivalent to the following.

\medskip

\noindent{\textbf{Conjecture 6.}} Suppose $\alpha>0$, $\mathbf w\in(-1,1)^k$ and $\mathbf f\in\HH^k$. Then $\widehat{D}_\alpha(\mathbf f, \mathbf w)\leq0$.

\medskip

\noindent(4) There is , yet another way of looking at \eqref{eqn: Dalpha 1}. Suppose $\mathbf A=[a_{ij}]$ is a semi-positive definite $k\times k$ matrix, $a_{ij}\in\HH$, and $\mathbf c\in\CC^k$ is such that each entry of $\mathbf f:=\mathbf c\mathbf A$ belongs to $\HH$. Then we can define $N(\mathbf c,\mathbf A)=\mathbf c\mathbf A\mathbf c^*$ and
\[
\widetilde{D}(\mathbf c,\mathbf A)=2\mathrm{Re}\sum_{i=1}^kc_i\overline{f_i\Log f_i}-\sum_{i,j=1}^kc_i\overline{c_j}a_{ij}\Log a_{ij}-N(\mathbf c,\mathbf A)\Log N(\mathbf c,\mathbf A).
\]
Then we can ask whether $\widetilde{D}(\mathbf c,\mathbf A)\leq0$. In the proofs of Theorem \ref{thm: all except one positive} and Theorem \ref{thm: two points}, we are secretly using this definition: we let the matrix $\mathbf A$ to vary from $\mathbf W_\alpha$ to a rank $1$ matrix. Then we used the fact that if $\mathbf A$ has rank $1$, then $\widetilde{D}(\mathbf c,\mathbf A)=0$. An interesting observation is the following: if we define $\widetilde{\mathbf A}=[\widetilde{a}_{ij}]$ to be the block matrix
\[
\widetilde{\mathbf A}=\begin{bmatrix}
	\mathbf c\mathbf A\mathbf c^*&\mathbf f\\
	\mathbf f^*&\mathbf A
\end{bmatrix},
\]
and $\widetilde{\mathbf c}=(-1,\mathbf c)$, then $\widetilde{\mathbf A}$ is semi-positive definite and $\widetilde{D}(\mathbf c,\mathbf A)=\widetilde{D}(\widetilde{\mathbf c},\widetilde{\mathbf A})$. Moreover, $\widetilde{D}(\widetilde{\mathbf c},\widetilde{\mathbf A})$ has the simple expression
\[
\widetilde{D}(\widetilde{\mathbf c},\widetilde{\mathbf A})=-\sum\widetilde{c}_i\overline{\widetilde{c}_j}\widetilde{a}_{ij}\Log\widetilde{a}_{ij}.
\]
In general, we ask the following question.

\medskip

\noindent\textbf{Question 7.} Suppose $\mathbf A=[a_{ij}]$ is semi-positive definite and $\mathrm{Re}a_{ij}\geq0$, $\forall i,j=1,\cdots,k$. Suppose $\mathbf c\in\CC^k$ is such that $\mathbf c\mathbf A=\mathbf 0$. Do we have
\[
\widetilde{D}(\mathbf c,\mathbf A)=-\sum_{i,j=1}^kc_i\overline{c_j}a_{ij}\Log a_{ij}\leq 0?
\]
\subsection{Numerical Evidences}
One of the advantages that Theorem \ref{thm: discrete formula} offers is that we can now test Question 3 using numerical methods. We have tested for a wide range of values of $\mathbf c$ and $\mathbf w$. We list a few graphs for the interested readers.

\noindent(1) In the proof of Theorem \ref{thm: all except one positive}, we showed that $D_\alpha(\mathbf c, \mathbf w)$ is non-increasing in $\alpha$ under the given conditions. However, Figure \ref{fig: alpha} shows that this is not always true.
\begin{figure}[htb]
	\centering\includegraphics[width=3.5in]{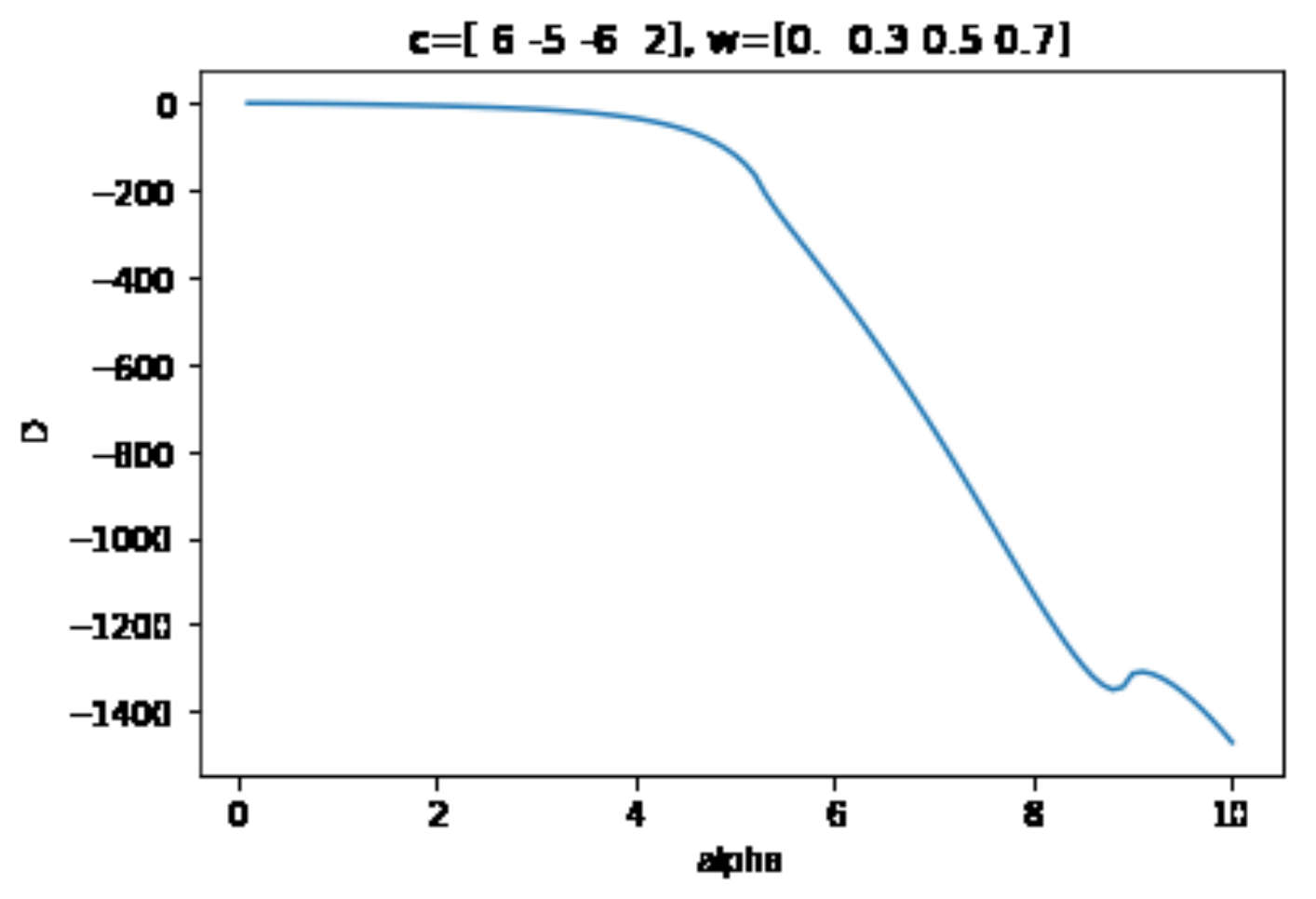}
	\caption{$k=4$, $\mathbf c, \mathbf w$ as indicated}\label{fig: alpha}
\end{figure}

\medskip

\noindent(2) If we adopt the definition $\widehat{D}_\alpha(\mathbf f,\mathbf w)$, then by fixing $f_2$ and $f_3$ and let $f_1$ vary, we get Figures \ref{fig: Re f1} and \ref{fig: Im f1}.

\begin{figure}[htb]
	\centering\includegraphics[width=3.5in]{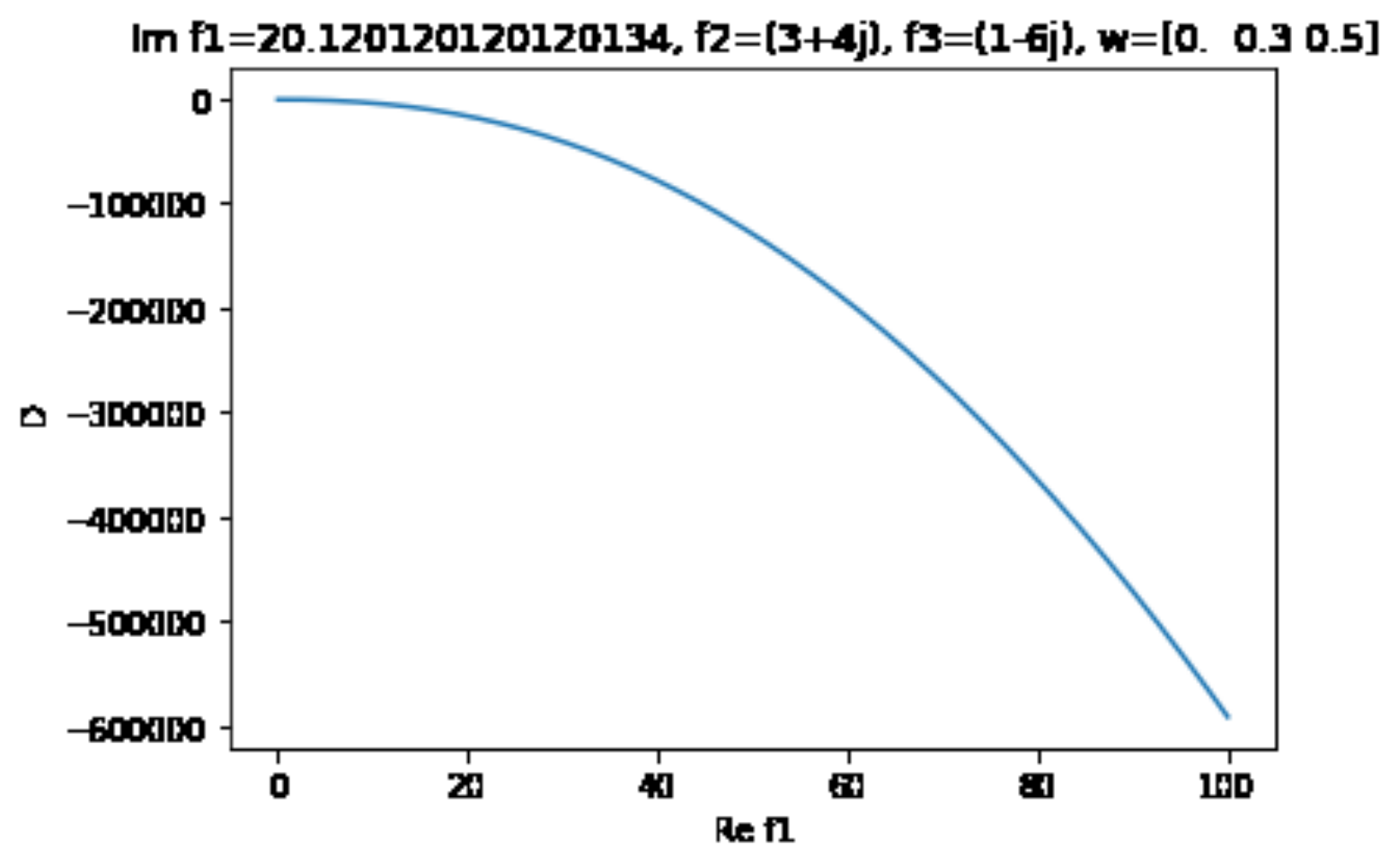}
	\caption{$k=3$, $\mathrm{Im}f_1$, $f_2$, $f_3$ and $\mathbf w$ as indicated}\label{fig: Re f1}	
\end{figure}

\begin{figure}[htb]	
	\centering\includegraphics[width=3.5in]{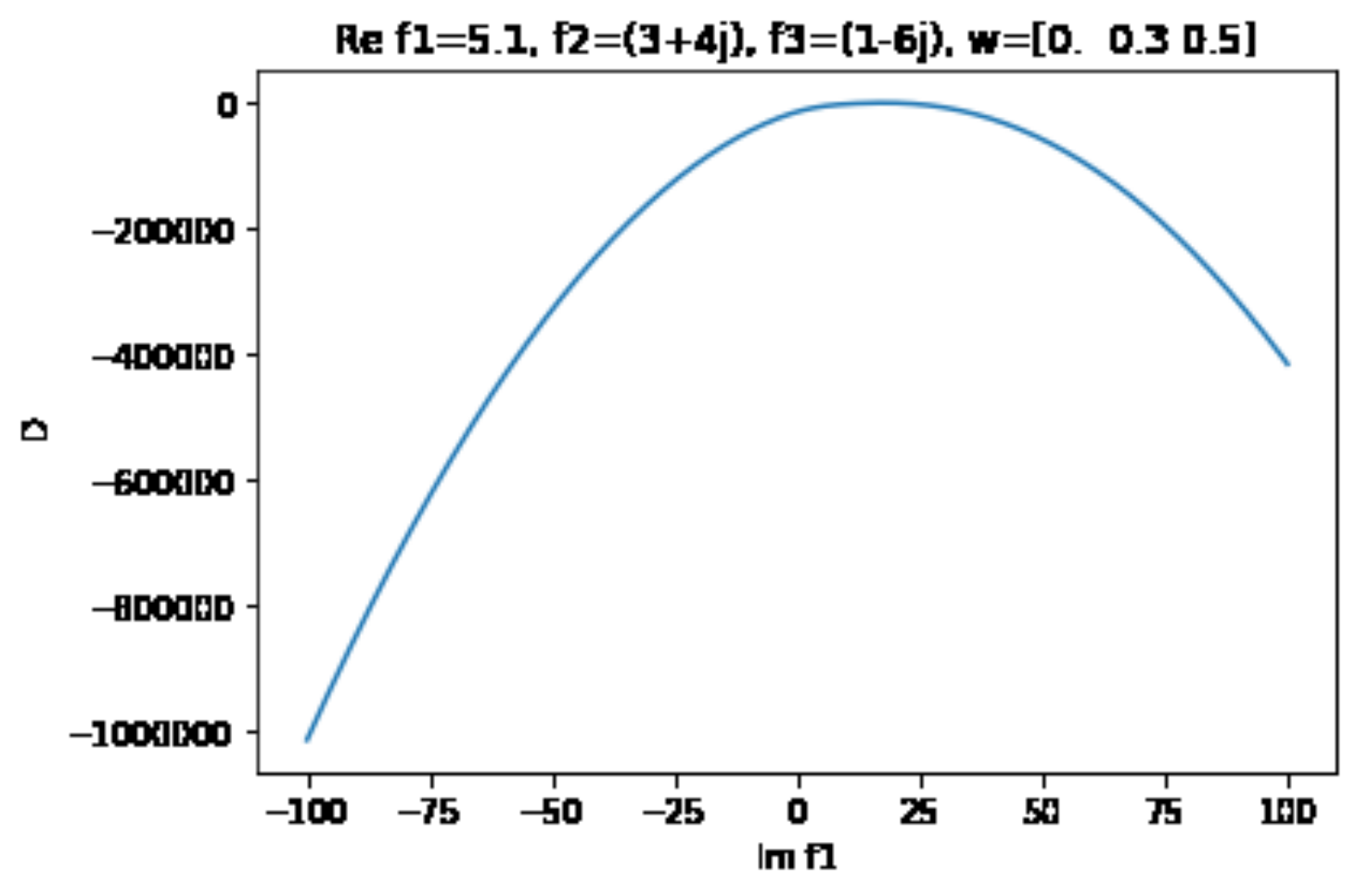}
	\caption{$k=3$, $\mathrm{Re}f_1$, $f_2$, $f_3$ and $\mathbf w$ as indicated}\label{fig: Im f1}
\end{figure}

\medskip

\noindent(3) In the special case when all entries of $\mathbf c$ are real, we can use the definition \eqref{eqn: Dalpha 3}. By letting one of the coefficients vary, we get Figure \ref{fig: c3}.

\begin{figure}[htb]
	\centering\includegraphics[width=3.5in]{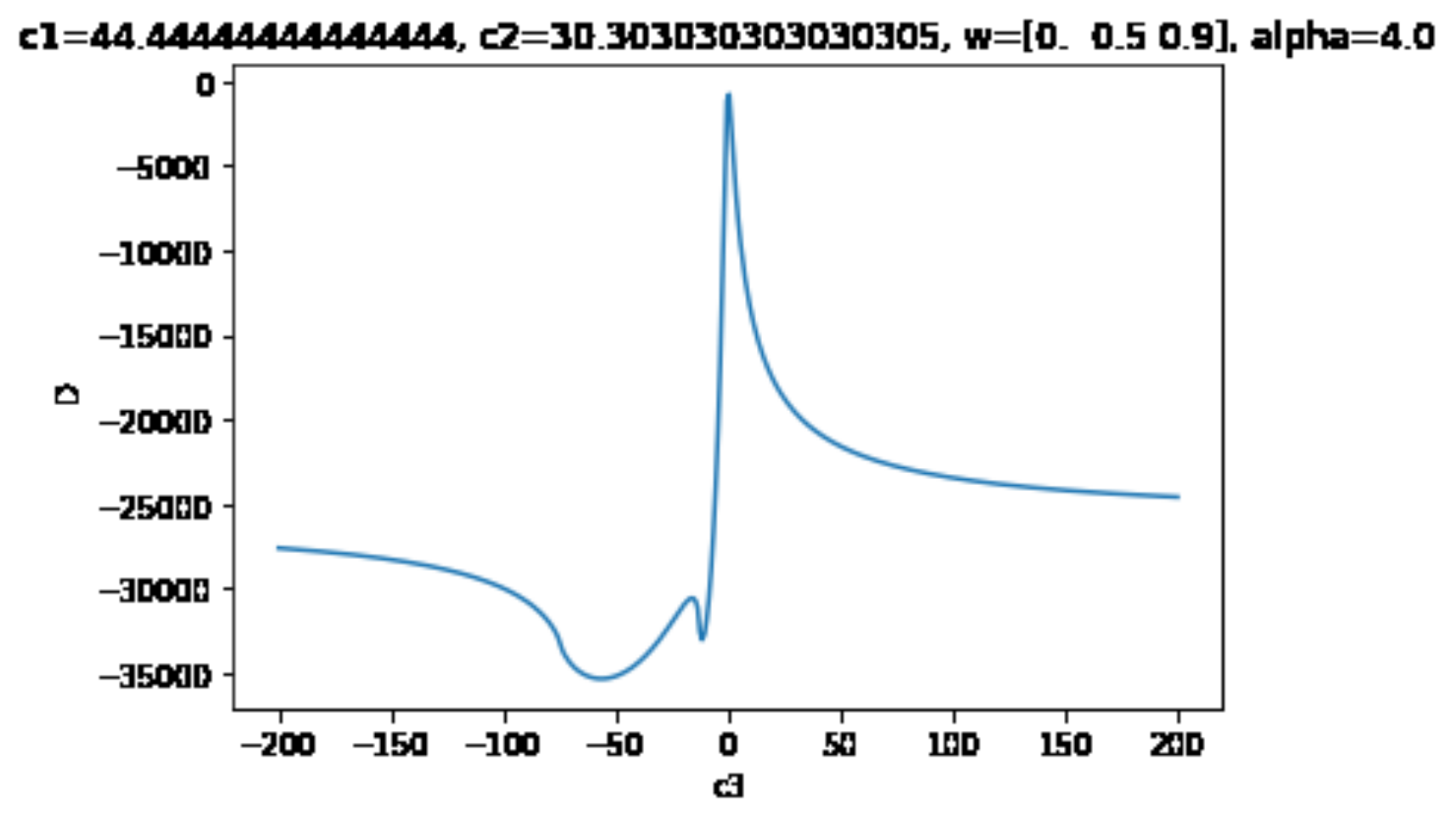}
	\caption{$k=3$, $c_1, c_2$ and $\mathbf w$ as indicated}\label{fig: c3}
\end{figure}

\medskip

\noindent\textbf{Acknowledgment:} The first and second author is partially supported by the Natural Science Foundation of China. The third author is partially supported by the Natural Science Foundation.


\vskip3mm \noindent{Hui Dan, School of Mathematical Sciences, Fudan
	University, Shanghai, 200433, China, E-mail: hdan@fudan.edu.cn

\noindent Kunyu Guo, School of Mathematical Sciences, Fudan
University, Shanghai, 200433, China, E-mail: kyguo@fudan.edu.cn

\noindent     Yi Wang,  Department of Mathematics, SUNY Buffalo, 244 Mathematics Building
Buffalo, NY 14260-2900
E-mail: yiwangfdu@gmail.com
 }
\end{document}